\documentclass[10pt]{amsart}

\usepackage[utf8]{inputenc}
\usepackage[T1]{fontenc}
\usepackage{textcomp}
\usepackage{graphicx}
\usepackage{amssymb,amsfonts,amsmath,amsthm}
\usepackage[english]{babel}
\usepackage{ae,aecompl} 
\usepackage{url}

\renewcommand{\leq}{\leqslant}
\renewcommand{\geq}{\geqslant}
\usepackage{indentfirst}

\newtheorem{theorem}{Theorem}
\newtheorem{proposition}[theorem]{Proposition}
\newtheorem{corollary}[theorem]{Corollary}
\newtheorem{lemma}[theorem]{Lemma}

\newcommand{\Fq}{\mathbb{F}_q}
\newcommand{\cl}[1]{\operatorname{cl}(#1)}
\newcommand{\ord}[1]{\mathcal{O}(#1)}
\newcommand{\E}{E}
\newcommand{\EndE}{\operatorname{End}(\E)}
\newcommand{\Ellt}{\operatorname{Ell}_t(\Fq)}
\newcommand{\Elltu}{\operatorname{Ell}_{t,u}(\Fq)}
\newcommand{\OrdE}{\ord{\E}}
\newcommand{\Zpi}{\mathbb Z\left[\pi\right]}
\newcommand{\kro}[2]{\left(#1\mid #2\right)}
\newcommand{\LL}[1]{L\left[#1\right]}
\begin{document}

\title[Computing endomorphism rings]{Computing the endomorphism~ring of an ordinary elliptic~curve over a finite field}
\author{Gaetan Bisson and Andrew V. Sutherland}

\begin{abstract}
We present two algorithms to compute the endomorphism ring of an ordinary elliptic curve $\E$ defined over a finite field $\Fq$.
Under suitable heuristic assumptions, both have subexponential complexity.
We bound the complexity of the first algorithm in terms of $\log{q}$,
while our bound for the second algorithm depends primarily on $\log |D_\E|$,
where $D_\E$ is the discriminant of the order isomorphic to $\EndE$.
As a byproduct, our method yields a short certificate that may be used to verify that the endomorphism ring is as claimed.
\end{abstract}
\maketitle
\section{Introduction}

Let $\E$ be an ordinary elliptic curve defined over a finite field $\Fq$, and let $\pi$ denote the Frobenius endomorphism of $\E$.
We may view $\pi$ as an element of norm $q$ in the integer ring of some imaginary quadratic field $K=\mathbb Q\left(\sqrt{D_K}\right)$:
\begin{equation}\label{equation:frob}
\pi=\frac{t+v\sqrt{D_K}}{2}\text{\quad with\quad}4q=t^2-v^2D_K.
\end{equation}
The trace of $\pi$ may be computed as $t=q+1-\#\E$.  Applying Schoof's algorithm to count the points on $\E/\Fq$, this can be done in polynomial time \cite{schoof-pointcounting}.
The fundamental discriminant $D_K$ and the integer $v$ are then obtained by factoring $4q-t^2$, which can be
accomplished probabilistically in subexponential time \cite{lenstra-factoring}.

The endomorphism ring of $\E$ is isomorphic to an order $\OrdE$ of $K$.
Once $v$ and $D_K$ are known, there are only finitely many possibilities for $\OrdE$, since
\begin{equation}\label{equation:endring}
\Zpi\subseteq\OrdE\subseteq\mathcal O_K.
\end{equation}
Here $\Zpi$ denotes the order generated by $\pi$, with discriminant $D_\pi=v^2D_K$,
and $\mathcal O_K$ is the maximal order of $K$ (its ring of integers), with discriminant $D_K$.
The discriminant of $\OrdE$ is then of the form $D_\E=u^2D_K$, where the conductor $u$ divides $v$ and uniquely determines $\OrdE$.
We wish to compute $u$.

\medskip

Recall that two elliptic curves over $\Fq$ are isogenous if and only if they have the same trace \cite[Ch.~13, Thm.~8.4]{husemoller-ec}.
Thus the set $\Ellt$ of elliptic curves defined over $\Fq$ with trace $t$ constitutes an isogeny class.
Each curve in $\Ellt$ has an endomorphism ring satisfying (\ref{equation:endring}), and therefore a conductor dividing $v$.

In his seminal thesis, Kohel describes the structure of the graph of isogenies defined on $\Ellt$,
and its relationship to the orders in $\mathcal{O}_K$.
He applies this structure to obtain a deterministic algorithm to compute $u$ in time $O(q^{1/3+\epsilon})$,
assuming the generalized Riemann hypothesis (GRH) \cite[Thm.~24]{kohel-thesis}.

Here we present two new methods to compute $u$ that further exploit the relationship between the isogeny graph and ideal class groups.
Under heuristic assumptions (including, but not limited to, the GRH), we achieve subexponential running times.
Both methods yield \emph{Las Vegas} algorithms: probabilistic algorithms whose output is unconditionally correct.  We rely on heuristic assumptions only to bound their expected running times.

In practice we find the algorithms perform well, and are able to handle problem sizes that were previously intractable.
We give computational examples over finite fields of cryptographic size where $v$ is large and not smooth (the most difficult case).
Over a 200-bit field, for example, the total running time is typically under an hour (see Section~\ref{section:examples} for details).

To express our complexity bounds, we adopt the usual notation
\[\LL{\alpha,c}\left(x\right)=\exp\left(\left(c+o(1)\right)\left(\log x\right)^\alpha\left(\log\log x\right)^{1-\alpha}\right).\]
Under the heuristic assumptions detailed in Section~\ref{section:complexity}, we derive the bound
\[\LL{1/2,\sqrt{3}/2}(q)\]
for Algorithm~1 (Corollary~\ref{cor-alg1}), and the bound
\[\LL{1/2+o(1),1}(|D_\E|)\thickspace + \thickspace \LL{1/3,c_f}(q)\]
for Algorithm~2 (Proposition~\ref{prop-alg2}).
The $\LL{1/3,c_f}$ term reflects the heuristic complexity of
factoring $4q-t^2$ using the number field sieve \cite{blp-nfs}. 
Algorithm~2 is slower than Algorithm~1 in general, but may be much faster when $u\ll v$.

In certain cryptographic applications the discriminant $D_\E$ is an important security parameter (see \cite{bisson-satoh} for one example), and
it may be necessary for a third party to independently verify its value.  The algorithms we use to compute $D_\E$ may additionally generate a short \emph{certificate} to aid this verification.
Both certification and verification have heuristically subexponential running times, and one may extend the certification phase in order to reduce the
verification time, as discussed in Section~\ref{section:complexity}.  Under the same heuristic assumptions used in our complexity bounds, the size of the certificate is $O(\log^{2+\epsilon}q)$
(Corollary~\ref{cor-cert}).

\section{Preliminaries}\label{section:preliminaries}

Kohel's algorithm treats each large prime power $p^k$ dividing $v$ by computing the kernel of a certain \emph{smooth isogeny} of degree $n$.  The prime factors of $n$ are small (polynomial in $\log v$), but $n$ itself is large (exponential in $\log v$), and this leads to an exponential running time (see \cite[Lem.~29]{kohel-thesis}).  We replace this computation with a walk in the isogeny graph using isogenies of low degree (heuristically, subexponential in $\log v$).  This walk computes the cardinality of a certain \emph{smooth relation}, and by performing similar computations in class groups of orders in $\mathcal O_K$ we are able to determine the power of $p$ dividing $u$ (via Corollary \ref{cor-piff}).  We adapt an algorithm of McCurley \cite{mccurley} to efficiently find smooth relations, achieving a heuristically subexponential running time.
First, we present some necessary background.

\subsection{Theoretical background}

Let us fix an ordinary elliptic curve $\E$ defined over a finite field $\Fq$, with $t$, $D_K$, and $v$
as in (\ref{equation:frob}).
We may verify that $\E$ is ordinary by checking that $t$ is nonzero modulo the characteristic of $\Fq$ \cite[Prop.~4.31]{washington-ec}.

Recall that the $j$-invariant $j(\E)$ may be computed as a rational function of the coefficients of $\E$ and, in particular, is an element of $\Fq$.
Over the algebraic closure of $\Fq$, the $j$-invariant uniquely identifies $\E$ up to isomorphism, but this is not true over $\Fq$.
However, two ordinary elliptic curves with the same trace are isomorphic over $\Fq$ if and only if they have the same $j$-invariant \cite[Prop.~14.19]{cox-primes}.
Thus we may explicitly represent the set $\Ellt$ as a subset of $\Fq$, namely, the $j$-invariants of all elliptic curves
over $\Fq$ with trace $t$, and view each element of $\Ellt$ as a particular elliptic curve representing its isomorphism class.

As noted above, each curve in $\Ellt$ has an associated $u$ dividing $v$ that identifies its endomorphism ring,
and we may partition $\Ellt$ into subsets $\Elltu$ accordingly.
We aim to distinguish the particular subset containing $\E$ by identifying relations that hold in some $\Elltu$ but not others.

Our main tool is the action of the ideal class group $\cl{u^2D_K}$ of $\ord{u^2D_K}$
(the order of $K$ with conductor $u$) on the set ${\rm Ell}_{t,u}(\Fq)$.  Here we rely
on standard results from the theory of complex multiplication, and the Deuring lifting theorem.

\begin{theorem}\label{theorem:action}
With $q$, $t$, $v$, and $D_K$ as in $(1)$, let $u$ be a divisor of $v$ and $\mathfrak{a}$ an ideal of $\ord{u^2D_K}$ with prime norm $\ell$.
Then $\mathfrak{a}$ acts on the set $\Elltu$ via an isogeny of degree $\ell$, and this defines a faithful group action by $\cl{u^2D_K}$.
\end{theorem}
For a proof, see Theorems~10.5, 13.12, and 13.14 in \cite{lang-ef}, or Chapter~3 of 
\cite{kohel-thesis}. For additional background, we also recommend \cite{cox-primes} and \cite[Ch.~II]{sileverman-ec2}.

Theorem~\ref{theorem:action} implies that the cardinality of $\Elltu$ is a multiple of the class number $h(u^2D_K)$, and in fact these values are equal \cite{serre-cm}.  In general, the curves $\ell$-isogenous to $\E$ need not belong to $\Elltu$.
However, when $\ell$ does not divide $v$, we have the following result of Kohel \cite[Prop.~23]{kohel-thesis}:

\begin{theorem}\label{top-volcano}
Let $\ell$ be a prime not dividing $v$.
There are exactly $1+\kro{D_\E}{\ell}$ isogenies
of degree $\ell$ starting from $\E$, and they all lead to curves with endomorphism ring isomorphic to $\OrdE$.
\end{theorem}

The notation $\kro{D_\E}{\ell}$ is the Kronecker symbol.
Note that $\kro{D_\E}{\ell}=\kro{D_K}{\ell}$, so we can compute it without knowing $D_\E$.
We are primarily interested in the case $\kro{D_\E}{\ell}=1$, where the prime $\ell$ splits
into distinct prime ideals of norm $\ell$ in $\OrdE$, and these ideals lie in inverse ideal classes $\alpha$ and $\alpha^{-1}$ in $\cl{D_\E}$
(if $\ell$ splits into principal ideals, then $\alpha=\alpha^{-1}=1$).
By Theorem~1, the orbit of $\E$ under the action of $\alpha$ corresponds to a cycle of $\ell$-isogenies whose length is equal to the order of $\alpha$ in $\cl{D_\E}$.
Additional details on the structure of the isogeny graph can be found in \cite{kohel-thesis} and, in a more concise way, in \cite{fouquet-morain}.


\subsection{Explicit computation}\label{section:explicit}

We implement class group computations using binary quadratic forms.
For a negative discriminant $D$, the ideals in $\ord{D}$ correspond to primitive, positive-definite,
binary quadratic forms $ax^2+bxy+cy^2$ (commonly noted $(a,b,c)$) with discriminant $D=b^2-4ac$.
The integer $a$ corresponds to the norm of the ideal.
Ideal classes in $\cl{D}$ are uniquely represented by reduced forms.
As typically implemented, the group operation has complexity $O(\log^2|D|)$ \cite{biehl-buchmann-forms}.\footnote{
The algorithm of \cite{schonhage-fastforms} has complexity $O(\log^{1+\epsilon}|D|)$, but we do not make use of it.}

To navigate the isogeny graph, we rely on the classical modular polynomial $\Phi_\ell(X,Y)$, which parametrizes pairs of $\ell$-isogenous elliptic curves.
This is a symmetric polynomial with integer coefficients.
For a prime $\ell$ not dividing $q$, two elliptic curves $\E_1$ and $\E_2$ defined over $\Fq$ are connected by an isogeny of degree $\ell$ if and only if $\Phi_\ell(j(\E_1),j(\E_2))=0$ \cite[Thm.~19]{washington-ec}.\footnote{This isogeny is necessarily cyclic, since it has prime degree.}

The polynomial $\Phi_\ell$ has size $O(\ell^3\log\ell)$ \cite{cohen-coeff-phi}, and may be computed in time $O(\ell^{3+\epsilon})$ \cite{enge-modular}.
When $\ell$ is small we use precomputed $\Phi_\ell\in\mathbb{Z}[X,Y]$, but for larger $\ell$ we compute $\Phi_\ell/\Fq$, that is, the integer polynomial $\Phi_\ell$ reduced modulo the characteristic of $\Fq$.
This can be accomplished in time $O(\ell^{3+\epsilon}\log{q})$ and space $O(\ell^{2+\epsilon}\log{q})$ using the CRT method described in \cite{bls-crt-modpoly}.  In practice one may consider alternative modular polynomials that are sparser and have smaller coefficients than $\Phi_\ell$.

To find the curves that are $\ell$-isogenous to $\E$, we compute the roots of the univariate polynomial $f(X) = \Phi_\ell(X,j(\E))$ in $\Fq$.
We may restrict ourselves to primes $\ell\nmid v$ with $\kro{D_\E}{\ell}=1$, so that $f(X)$ has exactly two roots, by Theorem~\ref{top-volcano}.
We find these roots by computing $\gcd(f,X^q-X)$ and solving the resulting quadratic, using an expected $O(\mathsf{M}(\ell)\log q)$ operations in $\Fq$ (this is the time to compute $X^q\bmod f$).  Given $\Phi_\ell/\Fq$, we use $O(\ell^2)$ operations in $\Fq$ to construct $f(X)= \Phi_\ell(X,j(\E))$.  For $\ell\gg \log{q}$ this dominates the time to find the roots of $f(X)$ and bounds the cost of taking a single step in the $\ell$-isogeny graph.

\subsection{Relations}
Let us suppose that $\alpha\in \cl{D_\E}$ contains an ideal of prime norm $\ell\nmid v$, and has order $e=|\alpha|$.
In this situation we say that the relation $\alpha^e=1$ holds in $\cl{D_\E}$.
We cannot actually compute $\alpha^e$ in $\cl{D_\E}$, since we do not yet know $D_\E$, but we may apply Theorem~\ref{theorem:action} to compute the action of either $\alpha^e$ or $\alpha^{-e}$ on $\E$ by walking a distance $e$ along the cycle of $\ell$-isogenies, starting from $j=j(\E)$.

\bigskip
\begin{tabular}{rl}
\multicolumn{2}{l}{\textbf{Algorithm} $\textsc{WalkCycle} (j,\ell,e)$:}
\\
1. & Set $j_0 \leftarrow j$.
\\
2. & Let $j_1$ be one of the two roots of $\Phi_\ell(X,j_0)$.
\\
3. & For $s$ from 1 to $e-1$:
\\
4. & \qquad Let $j_{s+1}$ be the root of $\Phi_\ell(X,j_s)/(X-j_{s-1})$.
\\
5. & Return $j_e$.
\end{tabular}
\bigskip

The roots of $\Phi_\ell(X,j_s)$ are typically distinct (exceptions require $|D_\E|\leq 4\ell^2$, by \cite[Thm.~2.2]{fouquet-morain}), but the algorithm applies in any case.

The choice of $j_1$ in Step~2 is arbitrary, it may correspond to the action of either $\alpha$ or $\alpha^{-1}$.
Nevertheless, since $e=|\alpha|=|\alpha^{-1}|$, we have $j_e=j_0$ in either case.
A difficulty arises when we consider a relation that is not unary,
say $\alpha_1^{e_1}\alpha_2^{e_2}=1$, where $\alpha_i$ contains an ideal of prime norm $\ell_i$ with $\ell_1\neq \ell_2$.
 Starting from $j(\E)$, we walk $e_1$ steps along the $\ell_1$-isogeny cycle, then walk $e_2$ steps along the $\ell_2$-isogeny cycle.
We must make two arbitrary choices and may compute the action of
$\alpha_1^{e_1}\alpha_2^{e_2}$, $\alpha_1^{e_1}\alpha_2^{-e_2}$, $\alpha_1^{-e_1}\alpha_2^{e_2}$ or $\alpha_1^{-e_1}\alpha_2^{-e_2}$.
The actions of these four elements are almost certainly not identical; even if $\alpha_1^{e_1}\alpha_2^{e_2}=1$ in $\cl{D_\E}$, it is unlikely that $\alpha_1^{e_1}\alpha_2^{-e_2}$ will fix $j(\E)$.

To address this situation, we formally define a \emph{relation} $R$ as a pair of vectors $(\ell_1,\ldots,\ell_k)$ and $(e_1,\ldots,e_k)$,
where each $\ell_i$ is prime, $\ell_i\nmid v$ and $\kro{D_K}{\ell_i}=1$, and each $e_i$ is a positive integer.%
\footnote{In practice, we may wish to relax the constraint $\ell_i\nmid v$ when $\ell_i$ is very small (e.g. 2), see \cite{sutherland-cm}.\label{small-div-v}}
The integer $k$ is the \emph{arity} of the relation.
Given a discriminant $D=u^2D_K$ with $u\mid v$,
choose ideal classes $\alpha_1,\ldots,\alpha_k\in\cl{D}$ so that $\alpha_i$ contains an ideal of norm $\ell_i$.
This ideal need not be the reduced representative of $\alpha_i$, and may be principal (implying $\alpha_i=1$), this depends on $D$.
We now define
\begin{equation}\label{equation:cardinality}
\#R/D:=\#\left\{\tau\in\left\{\pm 1\right\}^{\left\{1,\ldots,k\right\}}:
\prod_{i=1}^k\alpha_i^{\tau_ie_i}=1\right\},
\end{equation}
as the \emph{cardinality} of the relation $R$ in $\cl{D}$.
  When $\#R/D>0$, we say $R$ \emph{holds} in $\cl{D}$. The integer $\#R/D$ is independent of the choice of the $\alpha_i$.
It has even parity, since if $\tau$ belongs to the set in (\ref{equation:cardinality}), so does $-\tau$.

To compute $\#R/D_\E$, we enumerate the $2^k$ possible walks we may take in the isogeny graph, starting from $j(\E)$, considering all possible sign vectors $\tau$ (these walks typically form a tree in which each path from root to leaf has $k$ binary branch points).
By the symmetry noted above, we may fix $\tau_1=1$.

\bigskip
\begin{tabular}{rl}
\multicolumn{2}{l}{\textbf{Algorithm} $\textsc{CountRelation} (\E,R)$:}
\\
1. & Compute $j\leftarrow\textsc{WalkCycle} (j(\E),\ell_1,e_1)$ and let $J$ be the list $(j)$.
\\
2. & For $i$ from 2 to $k$:
\\
3. & \qquad Set $J'\leftarrow J$ and then set $J$ to the empty list.
\\
4. & \qquad For $j\in J'$:
\\
5. & \qquad \qquad Set $j_0,j'_0\leftarrow j$ and let $j_1$ and $j'_1$ be the two roots of $\Phi_{\ell_i}\left(X,j_0\right)$.
\\
6. & \qquad \qquad For $s$ from 1 to $e_i-1$:
\\
7. & \qquad \qquad \qquad Let $j_{s+1}$ be the root of $f(X)=\Phi_{\ell_i}\left(X,j_s\right)/(X-j_{s-1})$.
\\
8. & \qquad \qquad \qquad Let $j'_{s+1}$ be the root of $f(X)=\Phi_{\ell_i}\left(X,j'_s\right)/(X-j'_{s-1})$.
\\
9. & \qquad \qquad Append $j_{e_i}$ and $j'_{e_i}$ to $J$.
\\
10. & Return $2n$, where $n$ counts the occurrences of $j(\E)$ in $J$.
\end{tabular}
\bigskip

Given $\Phi_\ell/\Fq$, the complexity of Algorithm \textsc{CountRelation} is dominated by
\begin{equation}\label{equation:rel-cost}
\sum_{i=1}^k2^{i-1}e_iT(\ell_i),
\end{equation}
where $T(\ell)$ is the time to take a single step in the $\ell$-isogeny graph, which for large $\ell$ is bounded by $O(\ell^2)$ operations in $\Fq$, as noted above.
Our algorithms rely on \emph{smooth} relations in which $k$, $\ell_i$, and $e_i$ are all rather small:
in the first example of Section~\ref{section:examples} we have $k=10$, $\ell_i\leq 500$, and $e_i\leq 3000$.
As a practical optimization, we order the couples $(\ell_i,e_i)$ to minimize (\ref{equation:rel-cost}), using an estimate of $T(\ell)$.

Computing $\#R/D$ in $\cl{D}$ (where $D$ is known) is straightforward:
one computes the set in (\ref{equation:cardinality})
by evaluating products of powers in $\cl{D}$.
A total of $O(2^k+\sum\log{e_i})$ operations in the class group suffice (independent of the $\ell_i$).

\subsection{Probing class groups}\label{probe-cl}

We now consider how we may distinguish class groups of orders in $K$ by computing the cardinality of suitable relations.
We rely on the following lemma.

\begin{lemma}\label{lemma:rel-div}
Suppose $\ord{D_1}\subseteq \ord{D_2}$.
Then for every relation $R$ we have
\[\#R/D_1\leq\#R/D_2.\]
\end{lemma}
\begin{proof}
Let $\mathfrak{a}$ be an $\ord{D_1}$-ideal with norm prime to the conductor of $D_1$.
The map
\[\mathfrak{a}\mapsto \mathfrak{a} \ord{D_2}\]
induces a natural morphism of class groups.
It preserves norms (see \cite[Prop.~7.20]{cox-primes} for a proof
in the case $D_2$ is fundamental, from which one easily derives the general case)
and therefore transports relations from $\cl{D_1}$ to $\cl{D_2}$.
\end{proof}

\begin{corollary}\label{cor-piff}
Let $p^k$ be a prime power dividing $v$, and let $D_1=(v/p^j)^2D_K$ and $D_2=p^{2k}D_K$, where $j=\nu_p(v)-k+1$.
Suppose $\#R/D_1>\#R/D_2$ for some relation $R$, and let $D=u^2D_K$ where $u\mid v$.  Then
$p^k\mid u$ if and only if $\#R/D<\#R/D_1$.
\end{corollary}

Provided we have a suitable relation $R$ for each prime-power $p^k$ dividing $v$, we can apply the corollary
to $D=D_\E$ to determine the prime-power factorization of $u$, and hence the endomorphism ring of $\E$.
The computations of $\#R/D_1$ and $\#R/D_2$ are performed in the class groups $\cl{D_1}$ and $\cl{D_2}$, but the computation of $\#R/D_\E$ takes place
in the isogeny graph via the \textsc{CountRelation} algorithm.
Notice that we may replace $v$ in the corollary by any multiple of $u$ dividing $v$.

\begin{proposition}\label{prop-term}
For all primes $p>3$, there are infinitely many relations satisfying the assumptions of Corollary~\ref{cor-piff}.
\end{proposition}
\begin{proof}
Consider unary relations with $e_1=1$ and $\ell_1=\ell$, and denote them $R_{\ell}$.  The relation $R_{\ell}$ holds in $\cl{D}$ precisely when $\ell$ splits into principal ideals in $\ord{D}$.
For $i\in\{1,2\}$, let $S_i$ be the set of primes $\ell$ such that $R_\ell$ holds in $\cl{D_i}$.
We now show $S_1\setminus S_2$ is infinite, referring to material from \cite[Ch.~8,9]{cox-primes}.

The set $S_i$ is equal to the set of primes that split completely in the ring class field $L_i$ of $\ord{D_i}$, and recall that $L_i$ is a Galois extension of $\mathbb Q$ \cite[Lem.~9.3]{cox-primes}.
The Chebotarev density theorem asserts that $S_1$ and $S_2$
are infinite, and $S_1\setminus S_2$ is finite if and only if $L_1\subseteq L_2$ \cite[Thm.~8.19]{cox-primes}.

But $L_1$ cannot be contained in $L_2$, for $\ord{D_1}$ is not contained in $\ord{D_2}$.
Indeed, $p^k$ divides the conductor of $\ord{D_2}$ but not that of $\ord{D_1}$,
which implies that $p^k$ divides the conductor of $L_2$
but not that of $L_1$ (see \cite[Ex.~9.20--9.23]{cox-primes}).
\end{proof}

In practice, of course, there are many other relations satisfying the requirements of Corollary~\ref{cor-piff}.
Empirically, relations $R$ holding in $\cl{D_1}$ satisfy $\#R/D_1 > \#R/D_2$ with probability converging to $1$ as $p$ grows.
We will not attempt to prove this statement, but as a heuristic assume that this probability is at least bounded above zero, and
furthermore that this applies to relations that are smooth (as defined in Section~\ref{section:complexity}).
Note that, independent of this assumption, the above proposition guarantees that our algorithms are always able to terminate.

\section{Algorithms}

\subsection{Computing $\OrdE$ from above}

We now describe our first algorithm to compute $u$, the conductor of the order $\OrdE$ isomorphic to $\EndE$.  We rely on Algorithm $\textsc{FindRelation}(D_1,D_2)$, described in Section~\ref{section:findrel}, to obtain relations to which Corollary~\ref{cor-piff} may be applied.

For small primes $p$ dividing $v$, say all $p\leq B$ for some $B$, we can efficiently determine the largest prime power $p^k$ dividing $u$ by isogeny climbing, as described in \cite[Sec.~4.2]{kohel-thesis} and \cite[Sec.~4.1]{sutherland-cm}.  This yields an isogenous curve $\E'$ for which the conductor of $\ord{\E'}$ is $u'=u/p^k$, using $O(kp^2\log{q})$ operations in $\Fq$ (given $\Phi_p/\Fq$).

For simplicity, the algorithm below assumes that $v$ is not divisible by the square of a prime larger than $B$.  The modification to handle large primes whose square divides $v$ is straightforward but unlikely to be needed in practice.

\bigskip
\begin{tabular}{rl}
\multicolumn{2}{l}{\textbf{Algorithm~1} $\left(\E/\Fq\right)$:}
\\
1. & Let Schoof's algorithm compute the trace $t$ of $\E$, then determine $D_K$, $v$, 
\\ & and the prime factors of $v$, by factoring $4q-t^2=-v^2D_K$.
\\
2. & Select a bound $B$ and set $u\leftarrow 1$.
\\
3. & For each prime $p\leq B$ dividing $v$:
\\
4. & \qquad Determine the largest power of $p$ dividing $u$ by isogeny climbing,
\\ & \qquad then set $\E\leftarrow \E'$, remove all powers of $p$ from $v$, and update $u$.
\\
5. & For each prime $p>B$ that divides $v$:
\\
6. & \qquad Set $R\leftarrow \textsc{FindRelation}(D_1,D_2)$ with $D_1=(v^2/p^2)D_K$, $D_2=p^2D_K$.
\\
7. & \qquad Determine whether $p$ divides $u$ by checking if $\#R/D_\E < \#R/D_1$,
\\ & \qquad then update $u$ appropriately.
\\
8. & Return $u$.
\end{tabular}
\bigskip

The correctness of Algorithm~1 follows from Corollary~\ref{cor-piff}.
Its running time depends on $B$ and the complexity of \textsc{FindRelation}.
Using $B=\LL{1/2,1/\sqrt{12}}(q)$, we obtain in Section~\ref{section:complexity}
(Corollary~\ref{cor-alg1}) a heuristic bound of
\[\LL{1/2,\sqrt{3}/2}(q)\]
on the expected running time of Algorithm~1, using $\LL{1/2,1/\sqrt{3}}(q)$ space.

Note that the relations computed in Algorithm~1 only depend on the Frobenius trace $t$ of $\E$, not its endomorphism ring,
hence they may be reused to compute the endomorphism ring of any curve in the same isogeny class.
These relations also provide a means to subsequently verify the computation of $u$,
but for this purpose we may wish to specialize the relations to $u$, a task we now consider.

\subsection{Certifying $u$}

Let us suppose that a particular value $u$ is claimed as the conductor of $\OrdE$.
This may arise in a situation where $u$ is actually known, either via Algorithm~1 or from the
construction of $\E$ (say, by the CM method), but may also occur when one wishes to test a provisional value of $u$,
as we will do in Algorithm~2.  We first give an algorithm to construct a \emph{certificate} that may be used to efficiently check whether a given curve with trace $t$ in fact has endomorphism ring $\OrdE$ with conductor $u$ (equivalently, it allows one to test whether an element of $\Ellt$ lies in $\Elltu$).

The construction of this certificate depends only on $u$, $v$, and $D_K$ and does not require an elliptic curve as input.  Small prime factors of $u$ and $v$ may be removed
by isogeny climbing prior to calling \textsc{Certify}.

\bigskip
\begin{tabular}{rl}
\multicolumn{2}{l}{\textbf{Algorithm} $\textsc{Certify} (u,v,D_K)$:}
\\
1. & For each prime factor $p$ of $v/u$:
\\
2. &\qquad Set $R_p\leftarrow \textsc{FindRelation}(D_1,D_2)$ with $D_1=u^2D_K$ and $D_2=p^2D_K$.
\\
3. & For each prime factor $p$ of $u$:
\\
4. &\qquad Set $R_p\leftarrow \textsc{FindRelation}(D_1,D_2)$ with $D_1=(u^2/p^2)D_K$, $D_2=u^2D_K$.
\\
5. & Return $C=(u,v,D_K,\{R_p\}_{p\mid v})$.
\end{tabular}
\bigskip

The relations computed in Step~2 may verify that the actual value of $u$ divides the claimed value,
whereas the relations computed in Step~4 may verify that the claimed value of $u$ is not a proper divisor of $u$,
as shown by Algorithm \textsc{Verify}.

\bigskip
\begin{tabular}{rl}
\multicolumn{2}{l}{\textbf{Algorithm} $\textsc{Verify} (\E/\Fq,C)$:}
\\
1. & For each prime factor $p$ of $v/u$, verify that $\#R_p/\E > \#R_p/p^2D_K$.
\\
2. & For each prime factor $p$ of $u$, verify that $\#R_p/(u^2/p^2)D_K > \#R_p/\E$.
\\
3. & Return \textbf{\tt true} if all verifications succeed and \textbf{\tt false} otherwise.
\end{tabular}
\bigskip

In addition to the verification of $u$ above, one may also wish to verify that $v$ and $D_K$ are correct.
This may be accomplished in polynomial time if the trace $t$ and the factorizations of $v$ and $D_K$ are
included in the certificate. One may additionally wish to certify the primes in these factorizations \cite{atkin-morain-ecpp},
or the verifier may apply a polynomial-time primality test \cite{aks-primes}.  Assuming these values are correct, the conductor
of $\OrdE$ is equal to $u$ if and only if $\textsc{Verify}(\E,C)$ returns \textbf{\tt true}.
This statement does not depend on any unproven hypotheses.

The size of the certificate is unconditionally bounded by $O(\log^3{q})$, and under heuristic assumptions we obtain 
an $O(\log^{2+\epsilon}{q})$ bound (Corollary~\ref{cor-cert}).
Within this bound, certificates for primes dividing $v$ or $D_K$ can be included,
as each certificate requires $O(\log^{1+\epsilon}q)$ space and there are $O(\log{q})$ such primes.

The expected running times of \textsc{Certify} and \textsc{Verify} depend on a smoothness parameter $\mu$
used by \textsc{FindRelation}.  This parameter may be chosen to balance the cost of certification and verification,
as in Algorithm~2 below, or one may reduce the verification time by increasing the certification time.
See Proposition~\ref{prop-verify} and the discussion following for an analysis of this trade-off.

\subsection{Computing $\OrdE$ from below}

We now present a second algorithm to compute $u$, which may be much faster than Algorithm~1 if $u$ is small compared to $v$,
and is in general only slightly slower.  Our basic strategy is to examine each of the divisors $u_i$ of $v$ in order,
attempting to prove that $u=u_i$ is the conductor of $\OrdE$, by constructing a certificate and verifying it.
This only requires finding relations in class groups with discriminants whose absolute value is at most $|u^2D_k|$.

Typically $v$ has few divisors (almost always $O(\log^{\log 2}v)$ \cite[p.~265]{hardy-wright}),
in which case this basic strategy is quite effective.  However, in order to improve performance in the worst case,
we apply isogeny climbing to effectively remove prime factors from $v$ as we go, thereby reducing the number of $u$'s we must consider.
As above, we suppose $v$ is square-free for the sake of presentation.

\bigskip
\begin{tabular}{rl}
\multicolumn{2}{l}{\textbf{Algorithm~2} $\left(\E/\Fq\right)$:}
\\
1. & Let Schoof's algorithm compute the trace $t$ of $\E$, then determine $D_K$, $v$, 
\\ & and the prime factors of $v$ by factoring $v^2D_K=4q-t^2$.
\\
2. & Set $x\leftarrow 0$.
\\
3. & Set $w\leftarrow\max(1/3,x/2+1/\log{q})$.
\\
4. & For primes $p<\exp(\log^w{v})$:
\\
5. & \qquad Test whether $p\mid u$ by isogeny climbing, then set $\E\leftarrow \E'$ and $v\leftarrow v/p$.
\\
6. & For divisors $u$ of $v$ less than $\exp(\log^{2w}v)$:
\\
7. & \qquad If $\textsc{Verify}(\E, \textsc{Certify}(u,v,D_k))$ returns \textbf{\tt true}:
\\  & \qquad \qquad Return the product of $u$ and the primes determined in Step~5.
\\
8. & Set $x\leftarrow 2w$ and go to Step~3.
\\
\end{tabular}
\bigskip

The variable $w$ is used to bound the complexity of isogeny climbing using a known lower bound for $u$ that increases as the algorithm proceeds.
Initially we have no information about $u$ so we use the cost of the factorization computed in Step~1 to select $w$.

The running time of Algorithm~2 is analyzed in Section~\ref{section:complexity}, where the bound
\[\LL{1/2+o(1),1}(|D_\E|) + \LL{1/3,c_f}(q)\]
is obtained under suitable heuristic assumptions.  The same assumptions yield an $\LL{1/2+o(1),2/3}(|D_\E|)\log q$ space bound.


\subsection{Finding Relations}\label{section:findrel}

Given negative discriminants $D_1$ and $D_2$, we seek a relation $R$ satisfying $\#R/D_1 > \#R/D_2$.  We find such an $R$ by searching for a relation that holds in $\cl{D_1}$ and then testing this inequality.  As noted at the end of Section~\ref{section:preliminaries}, this test almost always succeeds, but if not we search for another relation.

To find relations that hold in $\cl{D_1}$, we adapt an algorithm of McCurley \cite{hafner-mccurley,mccurley}.  Fix a smoothness bound $B$, and for each prime $\ell \leq B$ with $\kro{D_1}{\ell}\neq -1$, let $f_\ell$ denote the \emph{primeform} with norm $\ell$.  By this we mean the binary quadratic form $(\ell,b_\ell,c_\ell)$ of discriminant $D_1$ with $b_\ell\geq 0$, which may be constructed via \cite[Alg.~3.3]{buchmann-bqf}.

We then generate random reduced forms by computing the product
\begin{equation}\label{equation:abc1}
(a,b,c)=\prod_\ell f_\ell^{x_\ell},
\end{equation}
where the $x_\ell$ are suitably constrained (and mostly zero).  If the prime factors of $a$ are bounded by $B$, say  $a=\prod_\ell \ell^{y_\ell}$,
then we may decompose $(a,b,c)$ as
\begin{equation}\label{equation:abc2}
(a,b,c)=\prod_\ell f_\ell^{\tau_\ell y_\ell},
\end{equation}
where for nonzero $y_\ell$, $\tau_\ell=\pm 1$ is defined by $b=\tau_\ell b_\ell\bmod 2\ell$.

Recall that $n=\sqrt{|D_1|/3}$ is an upper bound on the norm of a reduced imaginary quadratic form \cite[Ex.~5.14]{crandall-pomerance}.
Provided that $\prod_\ell\ell^{|x_\ell|}>n$, the decompositions in (\ref{equation:abc1}) and (\ref{equation:abc2}) must be different,
since $a\leq n$.  This yields a non-trivial relation with exponents $e_\ell=x_\ell-\tau_\ell y_\ell$.

\medskip

In order to minimize the cost of computing $\#R/D_\E$ (via \textsc{CountRelation}) for the relations we obtain, in addition to
bounding the primes $\ell$, we must also bound the exponents $e_\ell$, and especially the number of nonzero $e_\ell$, which determines the arity $k$ of $R$. To achieve this we require all but a constant number $k_0$ of the $x_\ell$ to be zero (we use $k_0=3$), and note that if we assume
$a$ is a random $B$-smooth integer in $[1,n]$, then we expect it to have approximately $2\log{n}/\log{B}$ distinct
prime factors.  In the unlikely event that $k$ significantly exceeds this expected value, we seek a different relation.

Having bounded $k$, the complexity of \textsc{CountRelation} then depends on the products $e_\ell T(\ell)$ appearing in (\ref{equation:rel-cost}).
For large $\ell$ we have $T(\ell)=O(\ell^2)$ operations in $\Fq$.
To make the products $e_\ell T(\ell)$ approximately equal we may use the bound $|x_\ell|\leq (B/\ell)^2$.
In practice we use a bound
\[|x_\ell|\leq (B/\ell)^\omega,\]
that better reflects the cost of $T(\ell)$ for moderate values of $\ell$ (we typically use $\omega\approx 1.6$);
this has no impact on our asymptotic analysis.

\medskip

The Canfield--Erd{\H o}s--Pomerance theorem \cite[Thm.~3.1]{canfield-erdos-pomerance} implies that if we sample
uniformly random integers from the interval $[1,n]$ until we find one that is $\LL{1/2,\mu}(n)$-smooth, our expected
sample size is $\LL{1/2,1/(2\mu)}(n)$, where the implied constants can all be made explicit.  This allows us to compute a lower
bound $m(B,n)$ on the number of random integers we must sample from $[1,n]$ in order to have a better than $50\%$ chance of
finding one that is $B$-smooth. 

We initially set $B=\LL{1/2,\mu}(n)$, for a suitably chosen $\mu$, and compute $m(B,n)$ on the heuristic
assumption that the norms of the forms we generate are about as likely to be $B$-smooth as
random integers in the interval $[1,n]$.\footnote{This is true for random forms, see \cite[Prop.~11.4.3]{buchmann-bqf}.}
In practice we find this to be the case, however, to account for the possibility that none of the elements generated
according to our constraints have $B$-smooth norms (or that none of the relations we find are suitable),
we increase the smoothness bound by a constant factor $r$ slightly greater than 1, if we fail to find a suitable
relation after testing $2m(B,n)$ elements.

\bigskip
\begin{tabular}{rl}
\multicolumn{2}{l}{\textbf{Algorithm} $\textsc{FindRelation} (D_1,D_2)$:}
\\
1. & Set $B=\LL{1/2,\mu}(n)$, where $n=\sqrt{|D_1|/3}$.
\\
2. & Compute primeforms $f_\ell$ for $\ell\leq B$.
\\
3. & Repeat $2m(B,n)$ times:
\\
4. & \qquad Pick random integers $x_\ell$ with $|x_\ell| \leq (B/\ell)^\omega$ such that at most $k_0$ 
\\ & \qquad of the $x_\ell$ are nonzero and $\prod_\ell\ell^{|x_\ell|} > n$.
\\
5. & \qquad Compute the reduced form $(a,b,c)=\prod_\ell f_\ell^{x_\ell}$.
\\
6. & \qquad If $a$ is $B$-smooth:
\\
7. & \qquad\qquad Let $R$ be the relation with $e_\ell=|x_\ell-\tau_\ell y_\ell|$ where $a=\prod_\ell \ell^{y_\ell}$,
\\ & \qquad\qquad and let $k$ be the arity of $R$.
\\
8. & \qquad\qquad If $k<(2/\mu)\log^{1/2}n$ and $\#R/D_1 > \#R/D_2$, then return $R$.
\\
9. & Set $B\leftarrow rB$ and go to Step~2.
\end{tabular}
\bigskip

As a practical optimization, we may choose not to generate completely new values for $x_\ell$ every time Step~4 is executed, instead
changing just one bit in one of the nonzero $x_\ell$.  This allows the form $(a,b,c)$ to be
computed in most cases with a single composition/reduction using a precomputed set of binary powers of the $f_\ell$.

To implement Step~6 one may use the elliptic curve factorization method (ECM) to probabilistically identify $B$-smooth integers in time $\LL{1/2,2}(B)=\LL{1/4,2\mu}(n)$, which effectively makes the cost of smoothness testing negligible within the precision of our subexponential complexity bounds.  A faster approach uses Bernstein's algorithm, which identifies the smooth numbers in a given list in essentially linear time \cite{djb-smoothparts}.  This does not change our complexity bounds and for the sake of simplicity we use ECM in our analysis.

In practice, the bound $B$ is quite small (under 1000 in both our examples), and
very little time is spent on smoothness testing.  In our implementation we used a combination of trial division and a restricted form of Bernstein's algorithm.


\section{Complexity Analysis}\label{section:complexity}

The complexity bounds derived below depend on the following heuristics:
\begin{enumerate}
\item
\textbf{Small primes.}
We assume the GRH.  The effective Chebotarev bounds of Lagarias and Odlyzko then imply that for all $x=\Omega(\log^{2+\epsilon}|D_K|)$ there are $\Omega(x/\log{x})$ primes less than $x$ that split in $\mathcal{O}_K$, where the implied constants are all effectively computable \cite[Thm.~1.1]{lagarias-odlyzko-chebotarev}.
\vspace{3pt}
\item
\textbf{Random norms.}
We assume that the norms of the reduced forms computed in Step~4 of \textsc{FindRelation} have approximately 
the distribution of random integers in $[1,n]$.  Under this assumption, we apply the Canfield--Erd{\H o}s--Pomerance theorem
to estimate the probability of generating a form whose norm is $B$-smooth.
\vspace{3pt}
\item
\textbf{Random relations.}
If $D_1=u_1^2D_K$ and $D_2=u_2^2D_K$ are sufficiently large discriminants with $u_2\nmid u_1$, and $R$ is a random
relation for which $\#R/D_1 > 0$, with $\ell_i$ and $e_i$ bounded as in \textsc{FindRelation}, then we assume that $\#R/D_1 > \#R/D_2$ with
probability bounded above zero.
\vspace{3pt}
\item
\textbf{Integer factorization.}
We assume that ECM finds a prime factor $p$ of an integer $n$ in expected time $\LL{1/2,2}(p)\log^2{n}$ \cite{lenstra-ecm}, and that the expected running time of the number field sieve is $\LL{1/3,c_f}(n)$ \cite{blp-nfs}.
\end{enumerate}

In the propositions and corollaries that follow, we use the shorthand (\textbf{H}) to indicate that we are assuming Heuristics~1--4 above.

\begin{proposition}\label{prop-findrel}
$(\boldsymbol{\rm H})$ $\textsc{FindRelation}(D_1,D_2)$ has expected running time
\[\LL{1/2,1/(\sqrt{8}\mu)}(|D_1|)\thickspace +\thickspace \LL{1/2,0}(|D_1|)\log^3|D_2|.\]
The output relation $R$ has norms $\ell_i$ bounded by $\LL{1/2,\mu/\sqrt{2}}(|D_1|)$,
exponents $e_i$ bounded by $\LL{1/2,\sqrt{2}\mu}(|D_1|)$, and arity $k<(2/\mu)\log^{1/2}|D_1|$.
\end{proposition}
\begin{proof}
Let $B=\LL{1/2,\mu}(n)=\LL{1/2,\mu/\sqrt{2}}(|D_1|)$, where $n=\sqrt{D_1/3}$.
By Heuristic~1, for sufficiently large $B$ there are $\Omega(\log{B})$ primes $\ell=O(\log^2{B})$ with $\kro{D_1}{\ell}=1$.
For these $\ell$, the value of $|x_\ell|$ may range up to $B^{\omega-\delta}$, for any $\delta>0$.  Thus there are more than $2m(B,n)=\LL{1/2,1/(2\mu)}$ distinct
elements that may be generated in Step~4, and with high probability at least $m(B,n)$ are.
So Heuristic~2 applies, and with probability greater than $1/2$ we generate at least one element with $B$-smooth norm each time Step~3 is executed.

Under Heuristic~2, the expected number $k$ of nonzero exponents $e_i$ is at most
\[k_0+2\log{n}/\log{B}=k_0+\frac{1}{\mu}\log^{1/2}|D_1|(\log\log|D_1|)^{-1/2},\]
since we expect a random $B$-smooth integer in $[1,n]$ to have $(2+o(1))\log{n}/\log{B}$ distinct prime factors
(this may be proven with the random bisection model of \cite{bach-thesis}).
This, together with Heuristic~3, ensures that when Step~8 is reached the algorithm terminates, with some constant probability greater than zero.
Thus we expect to reach Step~9 just $O(1)$ times, and the total number of forms $(a,b,c)$ generated by the algorithm during its execution is bounded
by $\LL{1/2,1/(2\mu)}(n)$.

For each form $(a,b,c)$, the algorithm tests whether $a$ is $B$-smooth in Step~6.  Applying ECM, under Heuristic~4
we identify a $B$-smooth integer (with high probability) in time $\LL{1/2,2}(B)=\LL{1/4,\sqrt{2\mu}}(n)$ \cite{lenstra-ecm}.  
This yields
\[\LL{1/2,1/(2\mu)}(n)=\LL{1/2,1/(\sqrt{8}\mu)}(|D_1|),\]
as a bound on the expected time spent finding relations.

The bounds on $k$, the $\ell_i$, and the $e_i$ are immediate.  We may bound the cost of computing $\#R/D_j$, for $j=\{1,2\}$, by
\[O\left(2^k\log(\max e_i)\log^2|D_j|\right) = O(2^k\log^{5/2+\epsilon}|D_j|)=\LL{1/2,0}(|D_1|)\log^3|D_j|.\]
The proposition follows.
\end{proof}

\begin{corollary}\label{cor-alg1}
$(\boldsymbol{\rm H})$ Algorithm~1 has expected running time $\LL{1/2,\sqrt{3}/2}(q)$.
\end{corollary}
\begin{proof}
We may compute $t$ in polynomial time with Schoof's algorithm, and under Heuristic~4 we factor $4q-t^2=-v^2D_K$ 
in expected time $\LL{1/3,c_f}(q)$.

We use $B=\LL{1/2,1/\sqrt{12}}(q)$ in Algorithm~1, and set $\mu=1/\sqrt{6}$ when calling \textsc{FindRelation}.  The cost of isogeny climbing, the calls to \textsc{FindRelation}, and the calls to \textsc{CountRelation} to compute $\#R/D_\E$ all have expected complexity  $\LL{1/2,\sqrt{3}/2}(q)$, including the cost of computing the required $\Phi_\ell/\Fq$.
Only $O(\log{q})$ iterations are required in Algorithm~1 (one for each $p\mid v$), which does not change the complexity bound.
\end{proof}

\begin{corollary}\label{cor-cert}
$(\boldsymbol{\rm H})$ Let $D_1=u^2D_K$ and $D_2=v^2D_K$.   The expected running time of $\textsc{Certify}(u,v,D_K)$ is within an $O(\log{v})$ factor of the expected complexity of $\textsc{FindRelation}(D_1,D_2)$.  The output certificate $C$ has size $O(\log^{1+\epsilon}|D_1|\log{v})$.
\end{corollary}
\begin{proof}
Algorithm \textsc{Certify} makes fewer than $O(\log{v})$ calls to \text{FindRelation} with
$|D_1|\leq|u^2D_K|$ and $|D_2|\leq|v^2D_K|$.  Applying the bounds of Proposition~\ref{prop-findrel} for $\ell_i$, $e_i$, and $k$,
each relation has size $O(\log|D_1|\log\log|D_1|)$.
\end{proof}

\begin{proposition}\label{prop-verify}
$(\boldsymbol{\rm H})$ Given a certificate $C$ produced by Algorithm \textsc{Certify} with parameter $\mu$ and an
elliptic curve $\E/\Fq$, Algorithm $\textsc{Verify}(\E/\Fq,C)$ has expected running time 
\[\LL{1/2,3\mu/\sqrt{2}}(|u^2D_K|)\log^{5/2}q.\]
\end{proposition}
\begin{proof}
The expected time to compute $\Phi_\ell/\Fq$ is $O(\ell^{3+\epsilon}\log^{1+\epsilon}{q})$ \cite{bls-crt-modpoly,enge-modular}.
By Proposition~\ref{prop-findrel}, each relation in the certificate contains $O(\log^{1/2}|D_1|)$ distinct $\ell_i$,
each bounded by $\LL{1/2,\mu/\sqrt{2}}(|u^2D_K|)$.  There are at most $O(\log{q})$ relations
in the certificate, yielding a total time of
\[\LL{1/2,3\mu/\sqrt{2}}(|u^2D_K|)\log^{5/2}{q},\]
to compute all the $\Phi_\ell/\Fq$ needed for verification.
The total cost of all calls to \textsc{CountRelation} may be bounded by
\[\LL{1/2,\sqrt{2}\mu}(|u^2D_K|)\log^{2+\epsilon}{q},\]
using fast multiplication in $\Fq$, which is dominated by the bound above.
\end{proof}

To balance the costs of verification and certification, one uses $\mu=1/\sqrt{6}$.  The verification time
may be reduced (and the certification time increased) by choosing a smaller $\mu$.
For example, with $\mu=1/\sqrt{18}$ the verification time is $\LL{1/2,1/2}(|u^2D_K|)$
and the certification time is $\LL{1/2,3/2}(|u^2D_K|)$, ignoring logarithmic factors in $q$.

\begin{proposition}\label{prop-alg2}
$(\boldsymbol{\rm H})$ Algorithm~2 has expected running time
\[\LL{1/2+o(1),1}(|D_\E|) + \LL{1/3,c_f}(q).\]
\end{proposition}
\begin{proof}
In Step~1 we compute $t$ in polynomial time and factor $-v^2D_K$ in expected time $\LL{1/3,c_f}(q)$, by Heuristic~4.
Let $\mu=1/\sqrt{6}$ in all the calls to \textsc{Certify}, in order to balance the cost of \textsc{Verify}.
The cost of each certification/verification performed in Step~7 is then bounded by $\LL{1/2,\sqrt{3}/2}(|D_\E|)\log^{5/2}{q}$,
according to Proposition~\ref{prop-verify},
since we never test a divisor of $v$ that is greater than the conductor $u$ of $D_\E$.  In Step~6, $v$ can contain
no prime factors less than $\exp(\log^w{v})$.  Thus the number of divisors is bounded by
\[
\binom{\log^{1-w}v}{\log^{2w-w}v}\leq
\left(\log^{1-w}v\right)^{\log^wv}=\LL{w,1}(v)=\LL{1/2+o(1),1}(|D_\E|)\text{.}
\]
In the rightmost equality we have used
\begin{equation}\label{equation-switch}
\log|D_\E|>\log{u}\geq \log^{2w-1/\log{q}}v \quad \Rightarrow \quad \log{v}\leq\log^{1/(2w-1/\log{q})}|D_\E|
\end{equation}
to express the bound in terms of $|D_\E|$, noting that
\[
w/(2w-1/\log{q}) = 1/2 + 1/(4w\log{q}-2)\text{,}
\]
where $w\geq 1/3$ and $q\to\infty$ as $|D_\E|\to\infty$.
The cost of Step~7 for all the divisors considered in a single
execution of Step~6 is bounded by $\LL{1/2+o(1),1}(|D_\E|)\log^{5/2}{q}$.  The algorithm may repeat
Step~6 up to $\log{q}$ times, but the cost of each iteration dominates all prior ones, so we have
bounded the total cost of Step~7.

The cost of isogeny climbing in Step~5 during the first iteration is bounded by
\[\exp\left(\left(3+o(1)\right)\log^{1/3}{v}\right)\log^{2+\epsilon}q=\LL{1/3,c_f}(q)\]
(for any $c_f$), and thereafter cannot exceed
\[\exp\left((3+\epsilon)\log^w{v}\right)\log^{2+\epsilon}{q}=\LL{w,1}(v)\log^3{q}=\LL{1/2+o(1),1}(|D_\E|)\log^{2+\epsilon}{q}.\]
Here we have again applied (\ref{equation-switch}), and the choice of the constant 1 (or any constant)
is justified by the fact that $3/(\log\log|D_\E|)^{1-w}\to 0$ as $|D_\E|\to\infty$.

To complete the proof, we note that if $\LL{1/2+o(1),1}(|D_\E|)\log^{5/2}{q}$ exceeds $\LL{1/3,c_f}(q)$
we may incorporate the $\log^{5/2}{q}$ factor into the $o(1)$ term.  Otherwise, the complexity is $\LL{1/3,c_f}(q)$, and the proposition holds in either case.
\end{proof}

In both Algorithms~1 and 2, the space is dominated by the size of the polynomials $\Phi_\ell/\Fq$. 
As noted in Section~\ref{section:explicit}, these can be computed in $O(\ell^{2+\epsilon}\log{q})$ space \cite{bls-crt-modpoly}.
Plugging in parameters from the complexity analysis above, and making the same heuristic assumptions,
we obtain an $\LL{1/2,1/\sqrt{3}}(q)$ space bound for Algorithm~1, and an $\LL{1/2+o(1),2/3}(|D_\E|)\log{q}$
space bound for Algorithm~2.

\section{Examples}\label{section:examples}

The rough timings we give here were achieved by a simple implementation
running on a single 2.4GHz Intel Q6600 core.  The algorithm \textsc{FindRelation} was implemented using the GNU C/C++ compiler \cite{gnu-c} and the GMP library \cite{gmp}, and for \textsc{CountRelation} we used a PARI/GP script \cite{pari-gp}.
We did not attempt to maximize performance, our purpose was simply
to demonstrate the practicality of the algorithms on some large inputs.
In a more careful implementation, constant factors would be substantially
improved and many steps could be parallelized.

\subsection{First Example}

We consider the elliptic curve $\E/\Fq$ with Weierstrass equation
$Y^2 = X^3 - 3X + c_\E$, where
\[\begin{array}{c}
c_\E=660897170071025494489036936911196131075522079970680898049528;
\\
q=1606938044258990275550812343206050075546550943415909014478299.
\end{array}\]
Its trace $t=212$ is computed by the Schoof--Elkies--Atkin algorithm in a few seconds
and, factoring $4q-t^2$, it is nearly instantaneous to retrieve $D_K=-7$ and
\[
v=2\cdot 127\cdot \underbrace{524287}_{p_1}\cdot \underbrace{7195777666870732918103}_{p_2}.
\]

Let us apply Algorithm~1 to compute the conductor $u$ of $\OrdE$.
First, we use isogeny climbing to handle small prime factors $p$ of $v$,
those for which $\Phi_p$ can be computed in reasonable time
(or, more likely, have already been precomputed); here, this means $2$ and $127$.
It takes roughly 20 seconds to compute $\Phi_{127}$ and isogeny climbing itself takes less than 2 seconds.
We find none of these primes divide $u$; hence $\E'=\E$ and we may now assume $v=p_1p_2$.

\medskip

For $p_1$ we set $D_1=(v/p_1)^2D_K$ and $D_2=p_1^2D_K$ as in Corollary~\ref{cor-piff}.
To find a relation satisfying this corollary, we use Algorithm $\textsc{FindRelation}(D_1,D_2)$
with the bound $B=500$.  Corollary~\ref{cor-alg1} uses $B=\LL{1/2,1/\sqrt{12}}(q)\approx 1900$,
but, taking into account constant factors in the complexity estimates, we find
experimentally that $B=500$ better balances the expected running time of \textsc{FindRelation} with that of computing $\#R/D_\E$.
The iteration bound $2m(B,n)=6\cdot 10^7$ has been evaluated via $m(B,n)=1/\rho(u)$ with $u=\log n/\log B\approx 8$ using
Table~1 of \cite{granville}, computed by Bernstein.

After about 20 minutes, \textsc{FindRelation} outputs the relation $R$ with
\[(\ell_i^{e_i})=(2^{2533},11^{752},29^2,37^{47},79^1,113^1,149^1,151^2,347^1,431^1),\]
for which $\#R/D_1=2$ and $\#R/D_2=0$.
Note that, as suggested by Footnote~\ref{small-div-v},
we make use of $\ell=2$ even though it divides $v$ (using the algorithm in \cite[Sec.~4.2]{sutherland-cm}).
Now, to evaluate $\#R/D_\E$ using Algorithm $\textsc{CountRelation}(\E,R)$,
we need to compute the required modular polynomials.
We use precomputed $\Phi_\ell$ for $\ell<100$, and for $\ell\geq 100$ apply the algorithm in \cite{bls-crt-modpoly};
$\Phi_{431}$ takes 5 minutes, $\Phi_{347}$ takes 3 minutes, and the others take less than a minute each.
Finally, $\#R/D_\E=0$ is evaluated in 6.5 minutes.  Since $\#R/D_\E < \#R/D_1$, we conclude from Corollary~\ref{cor-piff} that $p_1$ is a factor of $u$.

We now turn to $p_2$ and set $D_1=(v/p_2)^2D_K$ and $D_2=p_2^2D_K$
accordingly. The relation $R=(2^{23},11^5,43^1,71^2)$ is found almost instantly
by $\textsc{FindRelation}(D_1,D_2)$,
and we have $\#R/D_1=2$ and $\#R/D_2=0$.
$\textsc{CountRelation}(\E,R)$
computes $\#R/D_\E=2$ in 1.5 seconds,
proving that $p_2\nmid u$ (since $\#R/D_\E \nless \#R/D_1$).

\medskip

All in all, we have found the conductor $u=p_1$ of the elliptic curve $\E$
defined over a $200$-bit prime field in only slightly more than half an hour of computation.
The sizes of the primes $p_1$ and $p_2$
represents nearly a worst-case; if $p_2$ was 5 or 6 bits larger the remaining part of $v$ would be small enough
that one could more efficiently use a combination of isogeny climbing and Hilbert class polynomials to
determine $u$.

We note that, in this example, we could have used the invariant $\gamma_2=j^{1/3}$ (or other more favorable invariants \cite{bls-crt-modpoly,enge-modular}) in place of $j$,
allowing us to use modular polynomials in place of $\Phi_\ell$ that can be more quickly computed.
Doing so would let us increase the bound $B$ (reducing the time to find relations), and lead to an overall improvement in the running time.

\subsection{Second Example}

Consider now the elliptic curve $\E:Y^2 = X^3 - 3X + c_\E$
defined over the $255$-bit prime field $\Fq$ where
\[\begin{array}{rl}
c_\E= &
14262957895783764742987524732821199570\backslash\\
& 860243293007735537575027051453663494306;\\
q= &
50272551883931021408091448710235646749\backslash\\
& 904660980498576680086699865431843568847.\\
\end{array}\]
As above, we compute its trace $t=1200$ via the SEA algorithm
in about 10 seconds, and an easy factorization yields $D_K=-7$ and
\[
v=2\cdot 127\cdot
\underbrace{582509}_{p_1}\cdot
\underbrace{582511}_{p_2}\cdot
\underbrace{852857}_{p_3}\cdot
\underbrace{2305843009213693951}_{p_4}.
\]

Let us run Algorithm~2 to compute $\OrdE$.
We start with $w=1/3$, and first remove the
prime factors of $v$ less than $\exp(\log^{1/3}v)\approx 85$. As in Example~1, the constant factors make this a slight underestimate, and we are happy to increase this bound to include both 2 and 127, which we handle by isogeny climbing.
We find that neither of these divide $u$, and therefore $\E'=\E$, so we now assume that $v=p_1p_2p_3p_4$.

We then reach Step~6 and consider divisors $u$ of $v$
less than $\exp(\log^{2w}v)\approx 4\cdot 10^{8}$, namely $p_1$, $p_2$ and $p_3$.
Starting with $u\leftarrow p_1$, the certificate $C$ generated in Step~7
by $\textsc{Certify}(u,v,D_K)$ consists of
\[
R_{p_4}=R_{p_3}=R_{p_2}=(2^{41}, 11^{31}, 37^1) \text{ and } R_{p_1}=(11^1),
\]
and takes negligible time to compute.
The call to $\textsc{Verify}(\E/\Fq,C)$ takes one second and returns {\tt false},
proving that $u\neq p_1$.

Turning to $u=p_2$, $\textsc{Certify}(u,v,D_K)$ quickly outputs the certificate
\[
R_{p_4}=R_{p_3}=R_{p_1}=(2^{85}, 11^2, 23^5, 29^3) \text{ and } R_{p_2}=(11^1)
\]
and $\textsc{Verify}(\E/\Fq,C)$ returns {\tt false}
after 1.5 seconds of computation; so $u\neq p_2$.

We next consider $u=p_3$; the certificate used is
\[
R_{p_4}=R_{p_2}=R_{p_1}=(2^{239}, 11^1, 37^3) \text{ and }
R_{p_3}=(11^1).
\]
Computing and verifying this certificate takes about a second, and in this
case the verification succeeds, proving that $u=p_3$.

The total running time is less than 15 seconds, most of which is spent point-counting.
For comparison, it takes $\textsc{FindRelation}(p_2^2D,p_1^2D)$ nearly five minutes to output a relation, followed by a twenty-minute computation to evaluate its cardinality, 
demonstrating the advantage of Algorithm~2 over Algorithm~1 in this example.


\section*{Acknowledgments}

The authors express their gratitude to
Pierrick Gaudry, Takakazu Satoh, Daniel J. Bernstein and Tanja Lange
for their careful reading and helpful comments on an early version of this paper.

\bibliographystyle{amsplain}
\providecommand{\bysame}{\leavevmode\hbox to3em{\hrulefill}\thinspace}
\providecommand{\MR}{\relax\ifhmode\unskip\space\fi MR }
\providecommand{\MRhref}[2]{%
  \href{http://www.ams.org/mathscinet-getitem?mr=#1}{#2}
}
\providecommand{\href}[2]{#2}


\begin{thebibliography}{10}

\bibitem{aks-primes}
Manindra Agrawal, Neeraj Kayal, and Nitin Saxena, \emph{Primes is in {P}},
  Annals of Mathematics \textbf{160} (2004), 781--793.

\bibitem{atkin-morain-ecpp}
A.O.L. Atkin and Fran{\c c}ois Morain, \emph{Elliptic curves and primality
  proving}, Mathematics of Computation \textbf{61} (1993), 29--68.

\bibitem{bach-thesis}
Eric Bach, \emph{Analytic methods in the analysis and design of
  number-theoretic algorithms}, ACM Distinguished Dissertation 1984, MIT Press,
  1985.

\bibitem{djb-smoothparts}
Daniel~J. Bernstein, \emph{How to find smooth parts of integers}, 2004,
  \url{http://cr.yp.to/papers#smoothparts}.

\bibitem{biehl-buchmann-forms}
Ingrid Biehl and Johannes Buchmann, \emph{An analysis of the reduction
  algorithms for binary quadratic forms}, Voronoi's Impact on Modern Science
  (Peter Engel and Halyna~M. Syta, eds.), Institute of Mathematics, Kyiv, 1998,
  available at
  \url{http://www.cdc.informatik.tu-darmstadt.de/reports/TR/TI-97-26.ps.gz},
  pp.~71--98.

\bibitem{bisson-satoh}
Gaetan Bisson and Takakazu Satoh, \emph{More discriminants with the
  {Brezing-Weng} method}, Progress in Cryptology--{INDOCRYPT} 2008
  (Dipanwita~R. Chowdhury, Vincent Rijmen, and Abhijit Das, eds.), Lecture
  Notes in Computer Science, vol. 5365, Springer, 2008, pp.~389--399.

\bibitem{bls-crt-modpoly}
Reinier Br{\"o}ker, Kristen Lauter, and Andrew~V. Sutherland, \emph{Computing
  modular polynomials with the {CRT} method}, 2009, in preparation.

\bibitem{buchmann-bqf}
Johannes Buchmann and Ulrich Vollmer, \emph{Binary quadratic forms: An
  algorithmic approach}, Springer, 2007.

\bibitem{blp-nfs}
Joseph~P. Buhler, Hendrik~W. {Lenstra Jr.}, and Carl Pomerance, \emph{Factoring
  integers with the number field sieve}, The Development of the Number Field
  Sieve (Arjen~K. Lenstra and Hendrik~W. {Lenstra Jr.}, eds.), Lecture Notes in
  Mathematics, vol. 1554, Springer-Verlag, 1993.

\bibitem{canfield-erdos-pomerance}
E.~Rodney Canfield, Paul Erd{\H o}s, and Carl Pomerance, \emph{On a problem of
  {Oppenheim} concerning ``factorisatio numerorum''}, Journal of Number Theory
  \textbf{17} (1983), 1--28.

\bibitem{cohen-coeff-phi}
Paula Cohen, \emph{On the coefficients of the transformation polynomials for
  the elliptic modular function}, Mathematical Proceedings of the Cambridge
  Philosophical Society \textbf{95} (1984), 389--402.

\bibitem{cox-primes}
David~A. Cox, \emph{Primes of the form $x^2+ny^2$}, John Wiley \& Sons, 1989.

\bibitem{crandall-pomerance}
Richard Crandall and Carl Pomerance, \emph{Prime numbers: A computational
  perspective}, second ed., Springer, 2005.

\bibitem{enge-modular}
Andreas Enge, \emph{Computing modular polynomials in quasi-linear time}, to
  appear in Mathematics of Computation, arXiv preprint 0704.3177, 2007.

\bibitem{fouquet-morain}
Mireille Fouquet and Fran{\c c}ois Morain, \emph{Isogeny volcanoes and the
  {SEA} algorithm}, Algorithmic Number Theory Symposium--{ANTS V} (Claus Fieker
  and David~R. Kohel, eds.), Lecture Notes in Computer Science, vol. 2369,
  Springer, 2002, pp.~276--291.

\bibitem{gmp}
Torbj\"{o}rn {Granlund et al.}, \emph{{GNU} multiple precision arithmetic
  library}, September 2008, version 4.2.4, available at
  \url{http://gmplib.org/}.

\bibitem{granville}
Andrew Granville, \emph{Smooth numbers: computational number theory and
  beyond}, Algorithmic Number Theory: Lattices, Number Fields, Curves and
  Cryptography (Joseph~P. Buhler and Peter Stevenhagen, eds.), vol.~44, MSRI
  Publications, 2008, pp.~267--323.

\bibitem{hafner-mccurley}
James~L. Hafner and Kevin~S. McCurley, \emph{A rigorous subexponential
  algorithm for computing in class groups}, Journal of the American
  Mathematical Society \textbf{2} (1989), no.~4, 837--850.

\bibitem{hardy-wright}
Godfrey~H. Hardy and Edward~M. Wright, \emph{An introduction to the theory of
  numbers}, fifth ed., Oxford Science Publications, 1979.

\bibitem{husemoller-ec}
Dale Husem{\"o}ller, \emph{Elliptic curves}, Springer, 1987.

\bibitem{kohel-thesis}
David Kohel, \emph{Endomorphism rings of elliptic curves over finite fields},
  PhD thesis of the University of California at Berkeley, 1996.

\bibitem{lagarias-odlyzko-chebotarev}
Jeffrey~C. Lagarias and Andrew~M. Odlyzko, \emph{Effective versions of the
  {Chebotarev} density theorem}, Algebraic number fields: {$L$}-functions and
  {Galois} properties (Proc. Sympos., Univ. Durham, Durham, 1975), Acadmeic
  Press, 1977, pp.~409--464.

\bibitem{lang-ef}
Serge Lang, \emph{Elliptic functions}, second ed., Springer, 1987.

\bibitem{lenstra-ecm}
Hendrik~W. {Lenstra Jr.}, \emph{Factoring integers with elliptic curves},
  Annals of Mathematics \textbf{126} (1987), 649--673.

\bibitem{lenstra-factoring}
Hendrik~W. {Lenstra Jr.} and Carl Pomerance, \emph{A rigorous time bound for
  factoring integers}, Journal of the American Mathematical Society \textbf{5}
  (1992), no.~3, 483--516.

\bibitem{mccurley}
Kevin~S. McCurley, \emph{Cryptographic key distribution and computation in
  class groups}, NATO ASI on Number Theory and Applications (R.A. Mollin, ed.),
  C, vol. 265, Kluwer, 1989, pp.~459--479.

\bibitem{pari-gp}
The {PARI} group, \emph{{PARI/GP}}, 2008, version 2.3.4, available at
  \url{http://pari.math.u-bordeaux.fr/}.

\bibitem{schonhage-fastforms}
Arnold Sch{\"o}nhage, \emph{Fast reduction and composition of binary quadratic
  forms}, International Symposium on Symbolic and Algebraic
  Computation--{ISSAC}'91 (Stephen~M. Watt, ed.), ACM Press, 1991,
  pp.~128--133.

\bibitem{schoof-pointcounting}
Ren{\'e} Schoof, \emph{Counting points on elliptic curves over finite fields},
  Journal de Th{\'e}orie des Nombres de Bordeaux \textbf{7} (1995), 219--254.

\bibitem{serre-cm}
Jean-Pierre Serre, \emph{Complex multiplication}, Algebraic Number Theory
  (John~W.S. Cassels and Albrecht Fr{\"o}lich, eds.), Academic Press, 1967.

\bibitem{sileverman-ec2}
Joseph~H. Silverman, \emph{Advanced topics in the arithmetic of elliptic
  curves}, Springer, 1999.

\bibitem{gnu-c}
Richard {Stallman et al.}, \emph{{GNU} compiler collection}, August 2008,
  version 4.3.2, available at \url{http://gcc.gnu.org/}.

\bibitem{sutherland-cm}
Andrew~V. Sutherland, \emph{Computing {Hilbert} class polynomials with the
  {Chinese Remainder Theorem}}, 2009, \url{http://arxiv.org/abs/0903.2785}.

\bibitem{washington-ec}
Lawrence~C. Washington, \emph{Elliptic curves: Number theory and cryptography},
  second ed., CRC Press, 2008.

\end{thebibliography}

\end{document}